\documentclass[12pt]{amsart}
\usepackage[colorlinks=true, pdfstartview=FitV, linkcolor=blue, citecolor=blue, urlcolor=blue]{hyperref}
\usepackage{amssymb,amscd}
\calclayout 

\def\quot#1#2{#1/\!\!/#2}

\def\C{\mathbb {C}}

\def\A{\mathfrak {A}}

\def\NN{\mathcal N}
\def\OO{\mathcal O}

\def\NN{\mathcal N}

\def\Z{\mathbb Z}

\def\SL{\operatorname{SL}}

\def\inv{^{-1}}
\def\lie#1{{\mathfrak #1}}
\def\lieg{\lie g}

\def\phi{{\varphi}}

\def\O{\mathcal O}

\def\codim{{\operatorname{codim}}}
\def\pr{{\operatorname{pr}}}

\def\Aut{\operatorname{Aut}}

\def\alg{{\operatorname{alg}}}

\def\ql{{\operatorname{q\ell }}}
\renewcommand{\sl}{{\operatorname{/\!\!/}}}

\def\zd{\Z/d\Z}

\numberwithin{equation}{subsection}

\newtheorem{theorem}[subsection]{Theorem}
\newtheorem{lemma}[subsection]{Lemma}
\newtheorem{proposition}[subsection]{Proposition}
\newtheorem{corollary}[subsection]{Corollary}

\theoremstyle{definition}

\theoremstyle{remark}

\newtheorem{remark}[subsection]{Remark}

\newtheorem{example}[subsection]{Example}

\begin{document}
\title{A characterization of linearizability for holomorphic  $\C^*$-actions}
\author{Frank Kutzschebauch and Gerald W.~Schwarz}

\address{Frank Kutzschebauch, Institute of Mathematics, University of Bern, Sidlerstrasse 5, CH-3012 Bern, Switzerland}
\email{frank.kutzschebauch@math.unibe.ch}

\address{Gerald W.~Schwarz, Department of Mathematics, Brandeis University, Waltham MA 02454-9110, USA}
\email{schwarz@brandeis.edu}

\thanks{F.~Kutzschebauch was supported by Schweizerischer Nationalfonds grant 200021-116165.  G.~W.~Schwarz thanks the University of Bern for hospitality and financial support.} 



\begin{abstract}  
Let $G$ be a reductive complex Lie group acting holomorphically on  $X=\C^n$. The (holomorphic) Linearization Problem asks if there is a holomorphic change of coordinates on $\C^n$ such that the $G$-action becomes linear.  Equivalently,  is there  a $G$-equivariant  biholomorphism $\Phi\colon X\to V$ where $V$ is a $G$-module? 
There is an intrinsic stratification of the categorical quotient $\quot XG$, called the Luna stratification, where the strata are labeled by   isomorphism classes of representations of reductive subgroups of $G$. Suppose that there is a $\Phi$ as above. Then $\Phi$ induces a biholomorphism $\phi\colon \quot XG\to \quot VG$ which is stratified, i.e., the stratum of $\quot XG$ with a given label is sent isomorphically to the stratum of $\quot VG$ with the same label.

  The counterexamples to the Linearization Problem   construct an action of $G$ such that $\quot XG$ is not stratified biholomorphic to any $\quot VG$. Our main theorem shows that, for a reductive group $G$ with $\dim G\leq 1$, the existence of  a stratified biholomorphism of $\quot XG$ to some $\quot VG$ is not only necessary but also \emph{sufficient\/} for linearization. In fact, we do not have to assume that $X$ is biholomorphic to $\C^n$, only that $X$ is a Stein manifold. 
\end{abstract}

\maketitle
\section{Introduction}
The problem of linearizing the action  of a reductive group $G$ on $\C^n$ has attracted much attention both in the algebraic and holomorphic settings (\cite{Huckleberry}, \cite{KorasRussell}, \cite{ Kraft1996},\cite{Ku}). The first results were obtained   in the algebraic category. If $X$ is an affine $G$-variety, then the quotient is the affine variety $\quot XG$ with coordinate ring $\OO_\alg(X)^G$. 
An early high point is a consequence of Luna's slice theorem \cite{Luna}. Suppose that $\quot XG$ is a point and that $X$ is smooth and contractible. Then $X$ is algebraically $G$-isomorphic to a $G$-module. The structure theorem for the group of algebraic automorphisms of $\C^2$ shows that any action on $\C^2$ is linearizable \cite[Section 5]{ Kraft1996}.  As a consequence of a long series of results by many people, it was finally shown in \cite{KraftRussell} that an effective action of a positive dimensional $G$ on $\C^3$ is linearizable. The case of finite groups acting on $\C^3$ remains open.

The first counterexamples to the Algebraic Linearization Problem were constructed by Schwarz \cite{Schwarz1989} for $n\geq 4$.  His examples came from negative solutions to the equivariant Serre problem, i.e.,   there are algebraic $G$-vector bundles with base a $G$-module which are not  isomorphic to the trivial ones (those of the form $\pr\colon W\times W'\to W$ where $G$ acts diagonally on the $G$-modules $W$ and $W'$). It is interesting to note that in these counterexamples, the nonlinearizable actions may have the same stratified quotient  as a $G$-module. 

By the equivariant Oka principle of Heinzner and Kutzschebauch \cite{Heinzner-Kutzschebauch}, all holomorphic $G$-vector bundles over a $G$-module are trivial. Thus the algebraic counterexamples to linearization are not counterexamples in the holomorphic category. But Derksen and Kutzsche\-bauch \cite{Derksen-Kutzschebauch} showed that for $G$ nontrivial, there is an $N_G\in\mathbb N$ such that there are nonlinearizable holomorphic actions of $G$ on $\C^n$, for every $n\geq N_G$.  

Before saying more about the examples of \cite{Derksen-Kutzschebauch} we consider a more general problem. Let  $X$ and $Y$ be Stein $G$-manifolds. There are the categorical quotients $\quot XG$ and $\quot YG$. We have the Luna stratifications of $\quot XG$ and $\quot YG$ labeled by isomorphism classes of representations of reductive subgroups of $G$ (see Section \ref{sec:background}). We say that a biholomorphism  $\phi\colon \quot XG\to \quot YG$ is \emph{stratified\/} if it sends every  Luna stratum of $\quot XG$ with a given label to the Luna stratum of $\quot YG$ with the same label.  If $\Phi\colon X\to Y$ is a $G$-biholomorphism, then the induced mapping $\phi\colon \quot XG\to \quot YG$ is stratified. Given a stratified biholomorphism $\phi\colon \quot XG\to \quot YG$, it seems  difficult to find minimal criteria for there to exist a $G$-biholomorphism $\Phi\colon X\to Y$. Positive results so far  require   the existence of a $G$-equivariant diffeomorphism of $X$ and $Y$ or a special type  of  $G$-equivariant homeomorphism of $X$ and $Y$.   See \cite{KLS,KLSOka} for details.  

The counterexamples of Derksen and Kutzschebauch to linearization are $G$-actions on $X=\C^n$ whose   quotients are not isomorphic, via a stratified biholomorphism,  to any quotient $\quot VG$ where $V$ is a $G$-module.  However, 
jointly with F. L\'arusson, we have shown in \cite{KLSb} that, under a mild assumption (largeness, see \S \ref{sec:background}), this is  the only way to get a counterexample to linearization. To formulate this result let us recall some notions from \cite{KLSb}.
Suppose that we have a stratified biholomorphism $\phi\colon \quot XG\to \quot VG$ where $V$ is a $G$-module. Then we may identify $\quot XG$ and $\quot VG$ and call the common quotient $Z$. We have quotient mappings $\pi_X\colon X\to Z$ and $\pi_V\colon V\to Z$. Assume there is an open cover $\{U_i\}_{i\in I}$ of $Z$ and $G$-equivariant biholomorphisms $\Phi_i:\pi_X^{-1}(U_i)\to \pi_V^{-1}(U_i)$ over $U_i$ (meaning that $\Phi_i$ descends to the identity map of $U_i$).  We express the assumption by saying that $X$ and $V$ are \textit{locally $G$-biholomorphic over a common quotient}.   
Equivalently, our original $\phi\colon \quot XG\to \quot VG$ locally lifts to $G$-biholomorphisms of $X$ to $V$. From \cite{KLSb} we have:

\begin{theorem}
Suppose that $X$ is a Stein $G$-manifold, $V$ is a $G$-module and $X$ and $V$ are locally $G$-biholomorphic over a common quotient. Then $X$ and $V$ are $G$-biholomorphic.
\end{theorem}

In fact, the $G$-biholomorphism induces the identity on the common quotient.

\begin{theorem}\label{thm:main.old}
Suppose that $X$ is a Stein $G$-manifold and $V$ is a $G$-module satisfying the following conditions.
\begin{enumerate}
\item There is a stratified biholomorphism $\phi$ from $\quot XG$ to $\quot VG$.
\item  $V$  (equivalently, $X$) is large.
\end{enumerate}
Then, by perhaps changing $\phi$, one can arrange that $X$ and $V$ are locally $G$-biholomorphic over $\quot XG\simeq \quot VG$, hence $\phi$ lifts to a $G$-equivariant biholomorphism $\Phi\colon X\to V$.
\end{theorem}

In \cite[Sec.\ 5]{KLSb} we establish the following:
\begin{theorem}
Let $X$ be a Stein $G$-manifold, $V$ a $G$-module and $\phi\colon \quot XG\to \quot VG$ a stratified biholomorphism. Suppose that
\begin{enumerate}
\item $\dim \quot VG=1$ $($\cite{Jiang}$)$ or
\item $G=\SL_2(\C)$.
\end{enumerate}
Then the conclusion of Theorem \ref{thm:main.old} holds.
\end{theorem}

The main point of this paper is to show that  we can replace (1) or (2) above by the condition $\dim G\leq 1$.  

\begin{theorem}\label{thm:main} Let $G$ be reductive with $\dim G\leq 1$. Let  $X$ be a Stein $G$-manifold, let $V$ be a $G$-module and let   
 $\phi\colon \quot XG\to \quot  VG$ be  a stratified biholomorphism. Then, after perhaps changing $\phi$ by an automorphism of $\quot VG$, there is a biholomorphic $G$-equivariant lift $\Phi$ of $\phi$  from $X$ to $V$.
\end{theorem}

 Theorem \ref{thm:main} follows from Theorems \ref{thm:lift.finite}, \ref{thm:main.nonstable}, \ref{thm:two.pos.neg} and \ref{thm:last.case} below.

\smallskip\noindent
\textit{Acknowledgement.}  We thank the referees for corrections and for suggesting  improvements to the exposition. 
  \section {Background}\label{sec:background}

   For more information about the following see \cite{Luna} and \cite[Section~6]{Snow}.  Let $X$ be a normal Stein space with a holomorphic action of a reductive complex Lie group $G$.  The categorical quotient $Z=X\sl G$ of $X$ by the action of $G$ is the set of closed orbits in $X$ with a reduced Stein structure that makes the quotient map $\pi\colon X\to Z$ the universal $G$-invariant holomorphic map from $X$ to a Stein space.    Since $X$ is normal, $Z$ is normal.  If $U$ is an open subset of $Z$, then $\OO_X(\pi^{-1}(U))^G \simeq \OO_Z(U)$. We say that a subset of $X$ is \textit{$G$-saturated\/} if it is a union of fibers of $\pi$. If $X$ is the Stein space  associated to a normal affine variety $X'$ and $G$ acts algebraically on $X'$, then $Z$ is just the complex space corresponding to the affine  variety with coordinate ring  $\OO_\mathrm{alg}(X')^G$. 

If $Gx$ is a closed orbit, then the stabilizer (or isotropy group) $G_x$ is reductive.  We say that closed orbits $Gx$ and $G{y}$ have the same \textit{isotropy type} if $G_x$ is $G$-conjugate to $G_{y}$. Thus we get the \textit{isotropy type  stratification} of $Z$ with strata $Z_{(H)}$ indexed by conjugacy classes $(H)$ of reductive subgroups of $G$. The inverse image of $Z_{(H)}$ is denoted $X^{(H)}$. If $Z$ is irreducible, then there is an open dense stratum $Z_\pr$, the \emph{principal stratum}, and $X_\pr$ denotes the corresponding set of closed orbits in $X$, the set of  \emph{principal orbits}. We say that the $G$-action is \emph{stable\/} if the set of closed $G$-orbits is dense in $X$. Then $X_\pr=\pi\inv(Z_\pr)$.

Assume that $X$ is smooth and let $Gx$ be a closed orbit. Then we can consider the \emph{slice representation\/} which is the action of $G_x$ on $T_xX/T_x(Gx)$. We say that closed orbits $Gx$ and $Gy$ have the same \emph{slice type\/} if they have the same isotropy type and, after arranging that $G_x=G_y$, the slice representations are isomorphic representations of $G_x$. 
The slice type (Luna) strata are locally closed smooth subvarieties of $Z$. The   Luna stratification  is finer than the isotropy type stratification, but the Luna strata are unions of connected components of the isotropy type strata \cite[Proposition 1.2]{Schwarz1980}. Hence if the isotropy strata are connected, the Luna strata and isotropy type strata are the same. This occurs for the case of a $G$-module \cite[Lemma 5.5]{Schwarz1980}. 
 Alternatively, one can show directly that in a $G$-module, the isotropy group of a closed orbit determines the slice representation (see \cite[proof of Proposition 1.2]{Schwarz1980}).
 
Still assuming that  $X$ is smooth and that the quotient $Z$ is irreducible, we say that $X$ is \emph{$k$-principal\/}, $k\geq 0$,  if $X\setminus\pi\inv(Z_\pr)$ has codimension $k$ in $X$. Now assume that the $G$-action on $X$ is stable. Let $X_{(r)}$ denote the  set of  points   of $X$ whose isotropy group  has dimension $r$. We say that $X$ is \emph{$k$-modular\/}, $k\geq 0$, if $\codim_X\, X_{(r)}\geq r+k$,  $1\leq r\leq \dim G$. Note that this implies that there is an open dense set of closed $G$-orbits with finite isotropy group. We say that $X$ is $k$-large if it is $k$-principal   and $k$-modular. We use  \emph{large\/} as shorthand for  $2$-large. This is the key technical condition  in Theorem \ref{thm:main.old}. We have the following results about largeness \cite[\S  9 and \S  11]{GWSlifting}.

\begin{theorem}\label{thm:largeness}
Let $X$ be as above. Then 
\begin{enumerate}
\item $X$ is  large if and only if  every slice representation of $X$ is large.
\item Suppose  that $G$ is simple. There are only finitely many $G$-modules $V$ with $V^G=(0)$ which are not large.
\end{enumerate}
\end{theorem}
For $G$ semisimple, there is a slightly more complicated version of (2) \cite[Cor.\ 11.6]{GWSlifting}. Part (2) also holds with ``large'' replaced by ``$k$-large'' for any $k\geq 0$ \cite[Theorem 3.6]{HerbigSchwarzSeaton2}. Note that part (2) is false for tori. For example, if $G=\C^*$ and $V$ is any 1-dimensional  $G$-module, then $V$ is not large and there are   infinitely many non-isomorphic $V$. So, in the case of tori, there are infinitely many $V$ where Theorem \ref{thm:main.old} does not apply.   Thus our techniques to establish Theorem \ref{thm:main} are quite different than those we used to prove Theorem \ref{thm:main.old}.

\section{The infinitesimal lifting and deformation properties}\label{sec:ILP.DP}
 
Let $V$ be a $G$-module and let $\pi\colon V\to Z$ be the categorical quotient.  Let $\Aut(Z)$ denote the strata preserving holomorphic automorphisms of $Z$ and let $\A(Z)$ denote the holomorphic strata preserving vector fields on $Z$.   Let $\Aut(V)$  be the group of holomorphic automorphisms of $V$ and let $\A(V)$ denote the holomorphic vector fields on $V$.  Then we have a morphism $\pi_*\colon\A(V)^G\to\A(Z)$ by restricting $A\in\A(V)^G$ to $\O(Z)=\O(V)^G$.
Let $\Phi$ be a biholomorphism of Stein $G$-manifolds. We say that $\Phi$ is \emph{$\sigma$-equivariant\/} if there is an automorphism $\sigma\in\Aut(G)$ such that $\Phi\circ g=\sigma(g)\circ\Phi$ for all $g\in G$.   

  We say that $V$ has the \emph{infinitesimal lifting property} (ILP) if every strata preserving holomorphic vector field  $B\in\A(Z)$ is $\pi_*(A)$ for some  holomorphic vector field $A\in\A(V)^G$.  
 The scalar action of $\C^*$ on $V$ induces an action of $\C^*$ on $Z$ whose attractive fixed point is the image $*$ of $0\in V$. We denote by $\Aut_\ql(Z)$   the \emph{quasi-linear automorphisms of $Z$}, i.e., those which commute with the $\C^*$-action. If $\theta\in\Aut(Z)$, then we get a family $\theta_t\in\Aut(Z)$ where $\theta_t=t\inv\circ\theta\circ t$, $t\in\C^*$. We say that $V$ has the \emph{deformation property} (DP) if for any $\theta\in\Aut(Z)$, the limit $\theta_0$ of $\theta_t$ exists as $t\to 0$. Then $\theta_0\in\Aut_\ql(Z)$.
Finally, we say that $V$ has the \emph{lifting property} (LP) if any $\theta\in\Aut(Z)$ has a lift $\Theta$ to $V$. Here $\Theta$ need not be be equivariant, but, of course, it has to send the fiber over any $z\in Z$ to the fiber over $\theta(z)$.

 From Section 5 of \cite{KLSb} we have the following result.  
 
 \begin{theorem}\label{thm:easy}
 Let $X$ be a Stein $G$-manifold and let $\phi\colon\quot XG\to Z$ be a strata preserving biholomorphism. Suppose that $V$ has the ILP and DP. Then, after perhaps changing $\phi$ by an element of $\Aut_\ql(Z)$, $\phi$ has a $G$-equivariant biholomorphic lift $\Phi\colon X\to V$.
 \end{theorem}
 
 Here are some results about the ILP.
 
 \begin{proposition}\label{prop:ILP}
 Let $V$ be a $G$-module. Let $V_0$ be a $G$-complement to $V^G$ in $V$.
 \begin{enumerate}
 \item $V$ has the ILP if and only if $V_0$ has the ILP.
\item If $(V,G^0)$ has   the ILP then so does $(V,G)$.
\end{enumerate}
\end{proposition}

 \begin{proof}
See  \cite[Lemma 7.1(1), Proposition 8.2]{Schwarz1980}.  
 \end{proof}
 
 Let $\NN(V)=\pi\inv(\pi(0))$ denote the \emph{nullcone\/} of $V$. Let $H$ denote a \emph{principal isotropy group}, i.e., $(H)$ is the conjugacy class of the stratum $Z_\pr$. Then $Z$ is the quotient of $V^H$ by $N=N_G(H)/H$ where the principal isotropy groups of $V^H$ are trivial so that the action is stable \cite{LunaRichardson}.  The strata of $\quot VG$ are the same as those of $\quot {V^H}N$, with a change of label \cite[Thm.\ 11.3]{Schwarz1980}. Thus if $H$ is normal in $G$, then the ILP and/or the DP for $(V^H,N=G/H)$ implies that for $(V,G)$ (and vice versa).  

 \begin{proposition}\label{prop:ILP.tori}
Let $V$ be a $G$-module where $G^0=\C^*$  and $V^G=(0)$.  Assume  that $\dim Z\geq 3$ and that $\NN(V)$ has codimension at least $2$ in $V$.
Then $V$ has the ILP.
 \end{proposition}
 
 \begin{proof}
We may reduce to the case that the $G$-action is   faithful and that $G=\C^*$. Let   $B\in\A(Z)$ and let $*$ denote $\pi(0)$.  Let $C$ be the invariant vector field on $V$ corresponding to a generator of the Lie algebra $\lieg$ of $G$. Since the isotropy groups of closed nonzero orbits are finite, by Proposition \ref{prop:ILP}(2) and the slice theorem, there is an open cover $U_i$ of $Z\setminus \{*\}$ and lifts $A_i$ of $B$ to $\pi\inv(U_i)$. The differences $A_i-A_j$ are multiples of $C$ with coefficients in $\O(U_i\cap U_j)$. Thus the obstruction to glueing the $A_i$ is in $H^1(Z\setminus\{*\},\O_Z)$. Since $Z$ is Cohen-Macaulay of dimension at least 3, the obstruction vanishes and there is a lift  $A'$ of $B$ to $\A(V\setminus\NN(V))^G$. By Hartog's theorem, $A'$ extends to    $A\in\A(V)^G$ which is  a lift of $B$.
 \end{proof}

 \begin{proposition}\label{prop:LP.implies.DP}
If $V$ has the LP, then it has the  DP. 
 \end{proposition}
 \begin{proof}
This follows from  \cite[Proposition 2.9]{Schwarz2014}. (The proposition assumes that $V$ is $2$-principal, but that is not needed for the conclusion about the DP.)
\end{proof}

\begin{proposition}\label{prop:lift.finite}
Let $H$ be a finite group and $W$ an  $H$-module with quotient morphism $p\colon W\to W/H$. Let $U\subset W/H$ be open and let $\theta\in\Aut(U)$. Then there is a lift $\Theta\in\Aut(p\inv(U))$. Thus  $W$ has the LP, hence also the DP. 
\end{proposition}
\begin{proof}
See \cite[Theorem 5.4]{KrieglTensor} (see also \cite{Lyashko} or \cite[Theorem 3.1] {Schwarz2009}).
\end{proof}
\begin{remark}\label{rem:sigma.equivar}
Let $\theta\in\Aut(U)$ where $U$ is connected and let $\Theta\in\Aut(\pi\inv(U))$ be a lift of $\theta$. Then $\Theta$ sends $H$-orbits to $H$-orbits and for $h\in H$, $\Theta\circ h\circ\Theta\inv$ induces the identity on $U$, hence it must be an element $\sigma(h)$ of $H$ where $\sigma$ is an automorphism of $H$. Thus $\Theta$ is $\sigma$-equivariant.  
\end{remark}

\begin{corollary}\label{cor:DP.G0}
Let $Y=\quot V{G^0}$ and let $\rho\colon Y\to Z=\quot VG$ be the canonical map. Suppose that $Y$ has no codimension one strata. If $(V,G^0)$ has the LP and  $Y_\pr$ is simply connected, then $(V,G)$ has the LP (and DP).
\end{corollary}

\begin{proof}
Let $\theta\in\Aut(Z)$. Since $\rho$ is the quotient by a finite group, there is an open cover $U_i$ of $Y_\pr$ by $G/G^0$-invariant open sets  and $\theta_i\in\Aut( U_i)$ which cover $\theta$.  The compositions $\theta_i\circ\theta_j\inv$ take values in $G/G^0$, so we obtain a cocycle in $H^1(Y_\pr,G/G^0)$. This corresponds to a covering space of $Y_\pr$, which is trivial since $Y_\pr$ is simply connected. Hence changing the $\theta_i$ by elements of $G/G^0$, we construct a lift $\tilde\theta$ of $\theta$ to $Y_\pr$. Since $Y$ is normal and $Y\setminus Y_\pr$ has codimension at least two, $\tilde\theta$ extends to $Y$. Then $\tilde\theta$ preserves the strata of $Y$ since $\theta$ preserves the strata of $Z$.  Since $(V,G^0)$ has the LP, $\tilde\theta$ lifts to $\Theta\in\Aut(V)$ which is a lift of $\theta$.
\end{proof}

\begin{theorem}\label{thm:lift.finite}
Suppose that $G$ is finite and $\phi\colon X/G\to V/G$ is a strata preserving biholomorphism. Then, perhaps changing $\phi$ by an element of $\Aut_\ql(Z)$,  $\phi$ has a $G$-equivariant biholomorphic lift $\Phi\colon X\to V$.  The original $\phi$ has a $\sigma$-equivariant lift for some $\sigma\in\Aut(G)$.
\end{theorem}

\begin{proof}
This follows from Propositions \ref{prop:ILP} and \ref{prop:lift.finite}, Theorem \ref{thm:easy}  and Remark \ref{rem:sigma.equivar}.
\end{proof}

%%%%%%%%%%%%%%%%%%%%%%%%%%%%%%
\section{The easy cases}\label{sec:easy}
We return to the case where $G^0=\C^*$.
 In this section we consider the ``easy''  cases where  
 \begin{enumerate}
\item   all the weights of $G^0$ on $V$ have the same sign, or 
\item there are at least two weights of both signs.
\end{enumerate}
 Case (1) is not difficult and in Case (2)  we can manage to establish the ILP and DP  by eliminating codimension one strata in $\quot V{G^0}$.  In  the remaining ``hard'' case   of Section \ref{sec:hard}  there is only one positive weight or one negative weight and we have to argue quite differently.

\begin{theorem}\label{thm:main.nonstable}
If   the weights of $G^0$ on $V$ are all positive or   all negative, then  Theorem \ref{thm:main} holds.
\end{theorem}

\begin{proof}
All the closed $G$-orbits of $X$  lie in $X^{G^0}$ and similarly for $V$. Thus $\phi$ is a strata preserving biholomorphism of $X^{G^0}/G$ with $V^{G^0}/G$. By Theorem \ref{thm:lift.finite}, after modifying $\phi$ by an element of $\Aut(Z)$, there is an equivariant biholomorphic lift $\Phi\colon X^{G^0}\to V^{G^0}$. Let $V_0$ be a $G$-complement to $V^{G^0}$ in $V$.   By \cite{Heinzner-Kutzschebauch} the normal bundle of $X^{G^0}$ in $X$ is isomorphic to the normal bundle $V_0\times V^{G^0}\to V^{G^0}$ of $V^{G^0}$ in $V$. By \cite[Proposition 7]{KLSb} there is a $G$-biholomorphism $\Phi$ of a $G$-saturated neighborhood    of $X^{G^0}$ to a $G$-saturated neighborhood   of $V^{G^0}$ which induces $\phi$ on $X^{G^0}$. But a $G$-saturated neighborhood of $X^{G^0}$ is all of $X$ and similarly for $V^{G^0}$. Hence $\Phi\colon X\to V$ is a $G$-equivariant biholomorphic lift of $\phi$.
\end{proof}

\begin{theorem}\label{thm:two.pos.neg}
If    the action of $G^0$ on $V$ has at least two strictly positive weights and at least two strictly negative weights, then Theorem \ref{thm:main} holds.
\end{theorem}

\begin{proof} It is enough to establish the ILP and DP. For the ILP we may reduce to the case of $(V,G^0)$ and assume that $V^{G^0}=(0)$. Then the null cone $\NN(V)$  has codimension at least two in $V$ and by Propositions \ref{prop:ILP} and \ref{prop:ILP.tori}, $(V,G^0)$ and $(V,G)$ have  the ILP. To establish the LP (hence DP) it suffices to prove that $(V,G^0)$ has the LP and that $Y_\pr$ is simply connected, where $Y=\quot V{G^0}$ (the latter may be false, but we get around this). We may assume that $V^{G^0}=(0)$.

Let $p_1,\dots,p_k$ be the   positive weights of $V$ and let $q_1,\dots,q_l$ be the   negative weights with $k+l=n$. We may assume that the GCD of the $p_i$ and $q_j$ is $1$. We first deal with the codimension one strata of $Y$.  A codimension one stratum occurs whenever there are $n-1$ weights with GCD $r$ such that the remaining weight is prime to $r$. Thus the variable $x$ corresponding to the remaining weight, call it $d$, always occurs in any invariant to a power a multiple of $r$. Now let $F_r$ be   $\Z/r\Z\subset G^0$. Then the quotient of $V$ by $F_r$, call it $V'$, has the same weights as $V$, except that $d$ has been replaced by $rd$. This eliminates the corresponding codimension one stratum. We can do the same process for any other codimension one stratum of $Y$. Let $F$ be the product of all the $F_r$. Note that if $G^0$ is not in the center of $G$, then the weights of $V$ occur in pairs $\pm k$ in which case there are no codimension one strata. Hence if $F$ is not trivial, then $F$ is a normal subgroup of $G$.

\begin{example}
Suppose that $V$ has coordinate functions $\{x_1,\dots,x_4\}$  of weights $\{-2,-6,3,6\}$. Any invariant lies in $\C[x_1^3,x_2,x_3^2,x_4]$. Thus $F=\Z/6\Z\subset\C^*$ and $V'=V/F$ is $\C^4$ with weights $\pm 6$ each with multiplicity $2$.
\end{example}

Set $V/F=V'$ (considered as a vector space) and let $G'=G/F$. Then $Y'=\quot {V'}{(G^0/F)}$ has no codimension one strata. Dividing by a finite subgroup we can make the action of $G^0/F$ on $V'$ faithful. Then we have a fibration $\C^*\to V'_\pr\to Y'_\pr$. Since $V'_\pr$ is simply connected the exact homotopy sequence gives $0\to\pi_1(Y'_\pr)\to\pi_0(\C^*)$, hence $\pi_1(Y'_\pr)$ is trivial. Now we can apply Corollary \ref{cor:DP.G0} to establish the LP for $(V',G/F)$.

Now $Z'=\quot{V'} {G'}$ is the same variety as $Z=\quot V G$, but it has a coarser stratification since it is missing some codimension one strata. Let $\theta\in\Aut(Z)$. Then $\theta\in\Aut(Z')$ so it has a lift $\Theta'\in\Aut(V')$. Since $\theta$ preserves the  strata of $Z$, the  lift $\Theta'$ preserves the  strata of $V'$ considered as the quotient $V/F$. By Theorem \ref{thm:lift.finite}, $\Theta'$ lifts to $\Theta\in\Aut(V)$ and $(V,G)$ has the LP, hence the DP.
\end{proof}

%%%%%%%%%%%%%%%%%%%%%%%%%%%%%%%%%%%%%%%%%%%%%%%%
\section{The hard case}\label{sec:hard}
 
\begin{theorem}\label{thm:last.case}
 If $G^0$ acts with weights of both signs, but only one strictly positive weight or only one strictly negative weight, then Theorem \ref{thm:main} holds.
\end{theorem}
This  covers the remaining cases of Theorem \ref{thm:main}.
Without loss of generality we assume that there is one strictly positive weight and let $n$ be the number of strictly negative weights. It is easy to see that the ILP and DP hold in the case that $n=1$, so we assume that $n\geq 2$. If $g\in G$ induces a nontrivial automorphism of  $G^0$ then it would have to switch the positive and negative weight spaces of $V$, which is not possible. Hence $G$   centralizes $G^0$. 
 
 We need   two local versions of the LP, as follows.

\begin{proposition}\label{prop:local.lift}
Let $\theta\in \Aut(U) $ where $U$ is a connected neighborhood of $V^G$ in $Z$ and $\theta$ is the identity on $V^G$. Then, modulo $\Aut(Z)$, $\theta$ has a $G$-equivariant lift $\Theta$ to $\pi\inv(U )$.
\end{proposition}

\begin{proposition}\label{prop:local.lift.G0}
 Let $\theta\in \Aut(U) $ where $U$ is a connected neighborhood of $V^{G^0}/G$ in $Z$ and $\theta$ is the identity on $V^{G^0}/G$. Then, modulo $\Aut(Z)$, $\theta$ has a $G$-equivariant lift $\Theta$ to $\pi\inv(U)$.
\end{proposition}

\begin{proof}[Proof of Theorem \ref{thm:last.case}]
 Now $\phi$ induces a biholomorphic map of $X^G$ and $V^G$.  Let $V_0$ be a $G$-complement to $V^G$ in $V$. By \cite{Heinzner-Kutzschebauch} the normal bundle to $X^G$ in $X$ is equivariantly trivial, hence $G$-isomorphic to the normal bundle $V_0\times V^G\to V^G$ of $V^G$ in $V$. By \cite[Proposition 7]{KLSb}, there is a $G$-biholomorphism $\Psi\colon \Omega'\to \Omega$ where $\Omega'$ is a $G$-saturated neighborhood of $X^G$ in $X$ and $\Omega$ is a $G$-saturated neighborhood of $V^G$ in $V$. Moreover, $\Psi$ agrees with $\phi$ on $X^G$. We may assume that $U'=\quot {\Omega'}G$ and $U=\quot \Omega G$ are connected. Let $\theta=\phi\circ\psi\inv$ where $\psi\colon U'\to U$ is induced by $\Psi$. By Proposition \ref{prop:local.lift}  there is a $\tau\in\Aut(Z)$ such that $\tau\circ\theta$ lifts to $\Theta\in\Aut(\Omega)^G$. Then $\Phi_1:=\Theta\circ\Psi$ is an equivariant lift of $\tau\circ\phi$. Replace $\phi$ by $\tau\circ\phi$.

By Theorem \ref{thm:lift.finite} and Remark \ref{rem:sigma.equivar}, the restriction of $\phi$ to $X^{G^0}/G$ has a $\sigma$-equivariant lift $\phi_0\colon X^{G^0}\to V^{G^0}$ where $\sigma$ is an automorphism of $G/G^0$.  Now the restriction of $\Phi_1$ to $\Omega'\cap X^{G^0}$ is an equivariant lift of $\phi$. Thus $\phi_0$ and $\Phi_1$  locally differ on $\Omega'\cap X^{G^0}$ by multiplication by some $g\in G/G^0$. Then they differ on $X^{G^0}$ by a fixed $g\in G$. We can make $\phi_0$ equivariant by replacing it by $g\circ\phi_0$. Now by \cite{Heinzner-Kutzschebauch} and \cite[Proposition 7]{KLSb} again, we can find a   $G$-equivariant biholomorphism $\Psi \colon \Omega'\to \Omega$ which agrees with $\phi_0$ on $X^{G^0}$ where $\Omega'$ is a  $G$-saturated neighborhood of  $X^{G^0}$ in $X$ and  $\Omega$ is a  $G$-saturated neighborhood of $V^{G^0}$ in $V$. We may assume that $U_0'=\quot {\Omega'}G$ and $U=\quot{\Omega}G$ are connected. Let $\psi:U_0'\to U$ be induced by $\Psi$ and let $\theta=\phi\circ\psi\inv$. By Proposition \ref{prop:local.lift.G0}  there is a $\tau\in\Aut(Z)$  such that $\tau\circ\theta$ has a $G$-equivariant lift $\Theta$  on $\Omega'$.   Replace $\phi$ by $\tau\circ\phi$. Then $\Phi:=\Theta\circ\Psi$ is an equivariant lift of $\phi$ over  $U'_0$.

 Let $E$ be the Euler vector field on $V$ and let $\pi_*(E)$ denote its pushdown to $\A(Z)$. Let $B\in\A(\quot XG)$ denote the image of $\pi_*(E)$ using $\phi\inv$.  Using $\Phi$ one gets a lift $A_0\in \A(\pi\inv(U_0'))^G$ of $B$ restricted to $U_0'$. Since the isotropy groups of closed orbits outside of $X^{G^0}$ are finite, $B$ has local lifts $A_i$   over an open cover   $\{U_i'\}$ of $\quot XG$. Since $\quot XG$ is Stein, we see that there is no obstruction to finding a global lift $A\in\A(X)^G$ of $B$. 
Now one can use the argument for the proof of  Theorem 1 of \cite{KLSb} (see Remark 4 of loc.\ cit) to prove Theorem \ref{thm:last.case}.  \end{proof}

It remains to prove the two propositions. The rest of this section (and paper) is the proof of Proposition \ref{prop:local.lift.G0}. The proof of Proposition \ref{prop:local.lift} is similar. We may assume that the single strictly positive weight of $V$ is $d>0$ and that the other weights are negative.
We  write 
$$
V=\C\oplus V^G\oplus  W,\quad W=W_0\oplus W',\quad W'=\bigoplus_{i=1}^k W_i
$$ 
where $\C$ is the weight space of the weight $d$ and the weights in $W_i$  are $-e_i$,   $0=e_0<e_1<\cdots<e_k$. Let $x$ be a coordinate function on $\C$, let $x_1,\dots,x_p$ be coordinate functions on $V^G$, let $y_1,\dots,y_q$ be coordinate functions on $W_0$ and let $z_1,\dots,z_n$ be coordinate functions on $W'$ where each $z_i$ is a coordinate function on some $W_j$. Then $x$ has weight $-d$, the $x_i$ and $y_j$ have weight 0 and   the $z_i$ have weights in $\{ e_1,\dots,e_k\}$. Note that $G$ stabilizes $\C$, $V^G$ and the $W_i$, $i=0,\dots,k$.

Now let $H$ denote the isotropy group of $x_0:=1\in \C$. Then $H$ contains the isotropy group  $\zd$ of   $G^0$  at $x_0$. It is easy to see that $H/(\zd)\simeq G/G^0$. Since $G$ centralizes $G^0$, $H$ centralizes $\zd$.   Let $V'$ denote $V^G\oplus W$.

\begin{lemma}\label{lem:eval.x0}
Restriction to  $\{x_0\}\times V'$ gives an isomorphism  $\lambda\colon \C[V]^G\to\C[V']^H$.
\end{lemma}
\begin{proof}
We may assume that $V^G=(0)$. Clearly $\lambda$  is injective.   Let $f\in\C[W]^H$ and let $m$ be a monomial occurring in $f$. Then $m$ has degree $a_0,\dots,a_k$ in the variables of $W_0,\dots,W_k$. Thus $H\cdot m$ is a sum of monomials of the same degrees. Now $\sum a_i e_i\equiv 0\mod d$ so there is a unique $r$ such that $-rd+\sum a_ie_i=0$. Then $x^r(H\cdot m)\in\C[V]^G$ is $\lambda\inv(H\cdot m)$. Hence $\lambda$ is an isomorphism.
\end{proof}

As a consequence of the lemma, $Z=\quot VG\simeq Z':=V'/H$.  If $S=Z_{(L)}$ is a stratum of $Z$, then $S$ is the image of $V^{<L>}:=V^L\cap\pi\inv(S)$. We use similar notation for strata $Z'_{(L')}$ of $Z'$ and subsets $V'^{<L'>}$ of $V'$.

Let $p\colon V'\to Z'$ be the quotient mapping. Let $\theta\in\Aut(U)$ where $U$ is a connected neighborhood of $V^{G^0}/G$ in $Z$. Let $U'$ denote the subset of $Z'$ corresponding to $U$ and let
$\theta'$ denote the element of $\Aut(U')$ corresponding to $\theta$.  By Proposition \ref{prop:lift.finite} and Remark \ref{rem:sigma.equivar} there is a lift $\Theta'\in\Aut(p\inv(U'))$ which   is $\sigma$-equivariant for some automorphism $\sigma$ of $H$. Let $S=Z_{(L)}$ be a stratum of $Z$. If $L$ is finite, then we may assume that $L\subset H$. Now
$$
S=\pi(\C\setminus\{0\}\times V^G\times W^{<L>})=\pi(\{1\}\times V^G\times W^{<L>})=p(V^G\times W^{<L>})
$$
is a stratum of $Z'$. Moreover, $p\inv(S)$ is $\C^*$-stable (scalar action) and $\Theta'$ preserves $p\inv(S\cap U')$.  If $L$ is infinite, then $S$ is the image of
$$
V^{<L>}=V^G\times W_0^{<L>}=V^G\times W_0^{<L_0>}
$$
where $L_0$ is a subgroup of $H$ containing $\zd$ such that $L/\C^*\simeq L_0/(\zd)$. Thus $S$ is a subset of a stratum $S'$ of $Z'$. Now 
$p\inv(S)$ is $\C^*$-stable (scalar action) and $p\inv(S\cap U')$ is $\Theta'$-stable since $\Theta'$ induces $\theta$. Then $d\Theta'(0)$  is  $\sigma$-equivariant and induces   elements $\theta_0\in\Aut(Z)$ and $\theta_0'\in\Aut(Z')$. Replace $\theta$ by $\theta_0\inv\circ\theta$, replace $\theta'$ by $(\theta_0')\inv\circ\theta'$ and replace $\Theta'$ by
$d\Theta'(0)\inv\circ \Theta'$. Then $\Theta'\in\Aut(p\inv(U'))^H$ is a lift of $\theta'$.

Now $\Theta'=(x_1',\dots, x_p',y_1',\dots,y_q',z_1',\dots,z_n')$ where $x_1',\dots, x_p'$ are $H$-invariant, $(y_1',\dots,y_q')\colon V'\to W_0$ is $H$-equivariant where the $y_i'$ are $\zd$-invariant and 
 $z_i'$   has weight $e_j\mod d$ relative to $\zd$ where $z_i$ is a coordinate for $W_j$. 
The $x_i'$ are   the restrictions of  $a_1,\dots,a_p\in\O(\pi\inv(U))^G$.  The $y_i'$  lift to elements $b_1,\dots,b_q\in\O(\pi\inv(U))^{G^0}$ such that $(b_1,\dots,b_q)\colon \pi\inv(U)\to W_0$ is $G$-equivariant.  The main problem is that the $z_i'$ have lifts which may involve negative powers of $x$. It is clear that everything has holomorphic parameters coming from $V^G$, so we may reduce to the case that $V^G=(0)$. So $U'$ is now a neighborhood of $W_0/H$ in $Z'$   and $U$ is the corresponding open subset of $Z$. The morphism $\Theta'\colon p\inv(U')\to p\inv(U')$ is $H$-equivariant and preserves $W_0\times\{0\}$. Let $\Theta$ denote the (rational) lift of $\Theta'$ to $\pi\inv(U)$. We denote the entries of $\Theta$ by $(x,b_1,\dots,b_q,z_1'',\dots,z_n'')$.
 
 \begin{lemma}\label{lem:lifting.Theta'}
Let $\Theta$ be as above. Then there is a $\lambda\in\Aut(Z)$ and a lift $\Lambda$ of $\lambda$   such that $\Lambda\circ\Theta\in\Aut(\pi\inv(U))^G$.
 \end{lemma}

The lemma implies Proposition \ref{prop:local.lift.G0} and completes our proof of Theorem \ref{thm:last.case}, hence of Theorem \ref{thm:main}. The rest of the paper   is the proof of Lemma \ref{lem:lifting.Theta'}.
 
   Let $m$ be a monomial in $z_1,\dots,z_n$ which is a generator of the covariants of $\zd$ of weight $e_1\mod d$ as a module over the $\zd$-invariants. Then $m$ times a unique power of $x$ has weight $e_1$. The power of $x$ can be negative (see example below).
  Let $\overline{ M}$ be the $\C[W']^{\zd}$-module generated by the covariants of weights $e_1,\dots,e_k\mod d$ and let $M$ be the $(\C[x,x\inv]\otimes\C[ W'])^G$-module generated by the  corresponding (rational) covariants of weights $e_1,\dots,e_k$.
  
\begin{lemma}\label{lem:M}
Let $m\in M$ have weight $e_j$ and let $m_1,\dots,m_k$ be elements of $M$ of weights $e_1,\dots,e_k$. Then $m'=m(m_1,\dots,m_k)\in M$.
\end{lemma}
 
 \begin{proof}
Let $m\to\overline{ m}$ denote the mapping from $M$ to $\overline{M}$ given by restriction to $x=1$.    Then the mapping from the covariants of weight $e_j$ in $M$ to those of weight $e_j\mod d$ in $\overline{M}$ is an isomorphism. Since
$\overline{ m}'=\overline{m}(\overline {m}_1,\dots,\overline{m}_k)\in \overline M$, it follows that $m'\in M$.
 \end{proof}
 
 \begin{corollary}\label{cor:good.vars}
 Let $m\in M$. Then for any homogeneous element  $f\in\C[\C\oplus W']^G$ of sufficiently high degree, $fm$ is polynomial and remains polynomial after any substitutions $fm\mapsto fm(m_1,\dots,m_k)$.
 \end{corollary}
  	
\begin{example}	We have coordinate functions $(x,y,z,w)$ with weights $(-3,1,8,12)$. Mod 3 the weights on $(y,z,w)$-space are $(1,2,0)$. Thus, mod 3, the invariants are generated by $y^3$, $z^3$, $yz$ and $w$. The $G^0$-invariants are then generated by $xy^3$, $x^{8}z^3$, $x^3yz$ and $x^4w$. The isotropy groups of closed orbits are $G^0$, $\Z/3\Z$ and $\{e\}$.
	
	Now we look at the generators for the covariants for the $\Z/3\Z$-action on $(y,z,w)$-space. Those corresponding to the  representation where $\xi\in\Z/3\Z$ acts via multiplication   are $y$ and $z^2$. Those corresponding to the dual representation   are $y^2$ and $z$. Those corresponding to the trivial representation are $1$, $y^3$, $z^3$, $yz$ and $w$. Then taking into account the action of $G^0$ we see that $M$ is generated by the following covariants.
	
\begin{enumerate}
\item Weight 1: $y$ and $x^5z^2$.
\item Weight 8: $z$ and $x^{-2}y^2$.
\item Weight 12: $w$, $x^{-4}$, $x^{-3}y^3$, $x^4z^3$ and $x^{-1}yz$.
\end{enumerate}
The  elements of $M$ are closed under composition as per Lemma \ref{lem:M}.
\end{example}

 We call an element $m\in M$ \emph{good\/} if $m$ is polynomial and remains polynomial after any substitutions as in Lemma \ref{lem:M}, else we say that $m$ is \emph{bad}.  By Corollary \ref{cor:good.vars}, there are only finitely many homogeneous $m_i\in M$ which are bad.   We are going to modify $\Theta$ by a sequence of  meromorphic $G$-automorphisms $\Lambda$ of $V$ such that $\Lambda$ induces a strata preserving  automorphism   $\lambda\in \Aut(Z)$. 
 
 Let us assume that $z_1,\dots,z_\ell$ are the variables of $W_1$. Suppose that a monomial in the Taylor expansion of $z_1'$ (in the variables $z_1,\dots,z_n$) contains one of the variables  $z_1,\dots,z_\ell$. Then it is an invariant times the variable and is obviously good since invariants are preserved under composition as in Lemma \ref{lem:M}. Now we concentrate on the monomials which are not good.

Let $m_1,\dots,m_s$ be the bad monomials   of $\C[x,x\inv,z_1,\dots,z_n]$ of weight $e_1$.  The bad part of the Taylor series of  $z_j''$  is $\sum c_{ij} m_i$ for some $c_{ij}\in\O(W_0)$, $1\leq j\leq \ell$. We can assume that 
$$
(\sum_i c_{i1} m_{i },\dots,\sum_i c_{i\ell} m_i)\colon V\to  W_{1}
$$
 is $G$-equivariant since $\Theta$ is $G$-equivariant. Note that the monomials $m_i$ do not contain the variables $z_1,\dots,z_\ell$. Hence the $G$-automorphism (of de Jonqui\`ere type)
$$
\Lambda=(x,y_1,\dots,y_p,z_{1}-\sum_i c_{i1} m_i,\dots, z_\ell-\sum_i c_{i\ell} m_i, z_{\ell+1},\dots,z_n)
$$
  induces an automorphism  $\lambda$ of $Z$. The composition $\Lambda\circ\Theta$, by construction, has the property that the Taylor series of the entries $z_1'',\dots,z_\ell''$ contains only good terms. Now we can do the same thing to the Taylor series of $\Theta$ in the variables of $W_2$, etc. Hence there is a $\Lambda$ as claimed in Lemma \ref{lem:lifting.Theta'}.

  \bibliographystyle{amsalpha}
\bibliography{lin.paperbib}

\providecommand{\bysame}{\leavevmode\hbox to3em{\hrulefill}\thinspace}
\providecommand{\MR}{\relax\ifhmode\unskip\space\fi MR }
\providecommand{\MRhref}[2]{%
  \href{http://www.ams.org/mathscinet-getitem?mr=#1}{#2}
}
\providecommand{\href}[2]{#2}
\begin{thebibliography}{KLS17b}

\bibitem[DK98]{Derksen-Kutzschebauch}
Harm Derksen and Frank Kutzschebauch, \emph{Nonlinearizable holomorphic group
  actions}, Math. Ann. \textbf{311} (1998), no.~1, 41--53.

\bibitem[HK95]{Heinzner-Kutzschebauch}
Peter Heinzner and Frank Kutzschebauch, \emph{An equivariant version of
  {G}rauert's {O}ka principle}, Invent. Math. \textbf{119} (1995), no.~2,
  317--346.

\bibitem[HSS20]{HerbigSchwarzSeaton2}
Hans-Christian Herbig, Gerald~W. Schwarz, and Christopher Seaton,
  \emph{Symplectic quotients have symplectic singularities}, Compos. Math.
  (2020), no.~3, 613--646.

\bibitem[Huc90]{Huckleberry}
Alan~T. Huckleberry, \emph{Actions of groups of holomorphic transformations},
  Several complex variables, {VI}, Encyclopaedia Math. Sci., vol.~69, Springer,
  Berlin, 1990, pp.~143--196.

\bibitem[Jia92]{Jiang}
Mingchang Jiang, \emph{On the holomorphic linearization and equivariant {S}erre
  problem}, Ph.D. thesis, Brandeis University, 1992.

\bibitem[KLM03]{KrieglTensor}
Andreas Kriegl, Mark Losik, and Peter~W. Michor, \emph{Tensor fields and
  connections on holomorphic orbit spaces of finite groups}, J. Lie Theory
  \textbf{13} (2003), no.~2, 519--534.

\bibitem[KLS15]{KLS}
Frank Kutzschebauch, Finnur L{\'a}russon, and Gerald~W. Schwarz, \emph{An {O}ka
  principle for equivariant isomorphisms}, J. reine angew. Math. \textbf{706}
  (2015), 193--214.

\bibitem[KLS17a]{KLSOka}
Frank Kutzschebauch, Finnur L\'{a}russon, and Gerald~W. Schwarz, \emph{Homotopy
  principles for equivariant isomorphisms}, Trans. Amer. Math. Soc.
  \textbf{369} (2017), no.~10, 7251--7300. \MR{3683109}

\bibitem[KLS17b]{KLSb}
Frank Kutzschebauch, Finnur L\'arusson, and Gerald~W. Schwarz, \emph{Sufficient
  conditions for holomorphic linearisation}, Transform. Groups \textbf{22}
  (2017), no.~2, 475--485.

\bibitem[KR04]{KorasRussell}
Mariusz Koras and Peter Russell, \emph{Linearization problems}, Algebraic group
  actions and quotients, Hindawi Publ. Corp., Cairo, 2004, pp.~91--107.

\bibitem[KR14]{KraftRussell}
Hanspeter Kraft and Peter Russell, \emph{Families of group actions, generic
  isotriviality, and linearization}, Transform. Groups \textbf{19} (2014),
  no.~3, 779--792.

\bibitem[Kra96]{Kraft1996}
Hanspeter Kraft, \emph{Challenging problems on affine {$n$}-space},
  Ast\'erisque (1996), no.~237, Exp.\ No.\ 802, 5, 295--317, S{\'e}minaire
  Bourbaki, Vol. 1994/95.

\bibitem[Kut20]{Ku}
Frank Kutzschebauch, \emph{Manifolds with infinite dimensional group of
  holomorphic automorphisms and the linearization problem}, Handbook of Group
  Actions (Vol.V), ALM, vol.~48, Higher Education Press and International
  Press, Beijing-Boston, 2020, pp.~257--300.

\bibitem[LR79]{LunaRichardson}
Domingo Luna and Roger~W. Richardson, \emph{A generalization of the {C}hevalley
  restriction theorem}, Duke Math. J. \textbf{46} (1979), no.~3, 487--496.

\bibitem[Lun73]{Luna}
Domingo Luna, \emph{Slices \'etales}, Sur les groupes alg\'ebriques, Soc. Math.
  France, Paris, 1973, pp.~81--105. Bull. Soc. Math. France, Paris, M\'emoire
  33.

\bibitem[Lya83]{Lyashko}
Oleg~V. Lyashko, \emph{Geometry of bifurcation diagrams}, Current problems in
  mathematics, {V}ol. 22, Itogi Nauki i Tekhniki, Akad. Nauk SSSR Vsesoyuz.
  Inst. Nauchn. i Tekhn. Inform., Moscow, 1983, pp.~94--129.

\bibitem[Sch80]{Schwarz1980}
Gerald~W. Schwarz, \emph{Lifting smooth homotopies of orbit spaces}, Inst.
  Hautes \'Etudes Sci. Publ. Math. (1980), no.~51, 37--135.

\bibitem[Sch89]{Schwarz1989}
\bysame, \emph{Exotic algebraic group actions}, C. R. Acad. Sci. Paris S\'er. I
  Math. \textbf{309} (1989), no.~2, 89--94.

\bibitem[Sch95]{GWSlifting}
\bysame, \emph{Lifting differential operators from orbit spaces}, Ann. Sci.
  \'Ecole Norm. Sup. (4) \textbf{28} (1995), no.~3, 253--305.

\bibitem[Sch09]{Schwarz2009}
\bysame, \emph{Isomorphisms preserving invariants}, Geom. Dedicata \textbf{143}
  (2009), 1--6. \MR{2576288}

\bibitem[Sch14]{Schwarz2014}
\bysame, \emph{Quotients, automorphisms and differential operators}, J. Lond.
  Math. Soc. (2) \textbf{89} (2014), no.~1, 169--193.

\bibitem[Sno82]{Snow}
Dennis~M. Snow, \emph{Reductive group actions on {S}tein spaces}, Math. Ann.
  \textbf{259} (1982), no.~1, 79--97.

\end{thebibliography}

 \end{document}